\documentclass[12pt]{amsart}

\usepackage{epsfig}
\usepackage{verbatim}
\usepackage{amssymb}
\usepackage{amsfonts}
\newcommand{\be}{\beta}
\newcommand{\De}{\Delta}

\newcommand{\ga}{\gamma}
\newcommand{\e}{\varepsilon}
\newcommand{\Om}{\Omega}
\newcommand{\om}{\omega}

\newcommand{\BR}{\mathbb{R}}

\newcommand{\BN}{\mathbb{N}}
\newcommand{\A}{\mathcal{A}}

\newcommand{\T}{\mathcal T}

\newcommand{\R}{\mathfrak{R}}

\renewcommand{\L}{\mathcal{L}}

\renewcommand{\limsup}{\varlimsup}
\renewcommand{\liminf}{\varliminf}

\newcommand{\F}{\mathcal F}

\newcommand{\E}{\mathcal E}

\newcommand{\N}{\mathcal N}

\newtheorem{thm}{Theorem}[section]

\newtheorem{lem}[thm]{Lemma}
\newtheorem{pro}[thm]{Proposition}
\newtheorem{cor}[thm]{Corollary}
\newtheorem{ex}[thm]{Example}

\newtheorem{re}[thm]{Remark}
\def\R {\mathbb R}
\def\N {\mathbb N}

\def\cal{\mathcal}
\def\A{{\mathcal A}}
\def\L{{\mathcal L}}
\def\F{{\mathcal F}}

\begin{document}
\title[Growth rate for beta-expansions]
{Growth rate for beta-expansions}
\author{De-Jun Feng}
\address{
Department of Mathematics,
The Chinese University of Hong Kong,
Shatin,  Hong Kong, P. R. China}
\email{djfeng@math.cuhk.edu.hk}
\author{Nikita Sidorov}
\address{
School of Mathematics,
The University of Manchester,
Oxford Road,
Manchester M13 9PL,
United Kingdom}
\email{sidorov@manchester.ac.uk}
\date{\today}

\newpage
\begin{abstract}
Let $\be>1$ and let $m>\be$ be an integer. Each $x\in I_\be:=[0,\frac{m-1}{\be-1}]$ can be represented in the form
\[
x=\sum_{k=1}^\infty \e_k\be^{-k},
\]
where $\e_k\in\{0,1,\dots,m-1\}$ for all $k$ (a $\be$-expansion of $x$). It is known that a.e. $x\in I_\be$ has a continuum of distinct $\be$-expansions. In this paper we prove that if $\be$ is a Pisot number, then for a.e.~$x$ this continuum has one and the same growth rate. We also link this rate to the Lebesgue-generic local dimension for the Bernoulli convolution parametrized by $\be$.

When $\be<\frac{1+\sqrt5}2$, we show that the set of $\be$-expansions grows exponentially for every internal $x$.

\end{abstract}

\maketitle

\setcounter{section}{0}

\section{Introduction}
\setcounter{equation}{0}

Let $\be>1$ and let $m>\be$ be an integer. Put $I_\be=[0,(m-1)/(\be-1)]$. As is well known, each $x\in I_\be$ can be represented as a {\em $\beta$-expansion}
\[
x=\sum_{n=1}^\infty \e_n\be^{-n},\quad \e_n\in\{0,1,\dots,m-1\}.
\]
Since we do not impose any extra restrictions on the ``digits'' $\e_n$, one might expect a typical $x$ to have multiple $\be$-expansions. Indeed, it was shown that a.e. $x\in I_\be$ has $2^{\aleph_0}$ such expansions -- see \cite{S, DdV, S-nonl}.

The main purpose of this paper is to study the rate of growth of the set of $\be$-expansions for a generic $x$ when $\be$ is a Pisot number (see below). We also show that if $\be$ is smaller than the golden ratio, then every $x$, except the endpoints, has a continuum of $\be$-expansions with an exponential growth.

Now we are ready to state main results of this paper. Put
\begin{align*}
\E_n(x;\be)=\Bigl\{&(\e_1,\dots, \e_n)\in\{0,1,\dots,m-1\}^n \mid \exists
(\e_{n+1},\e_{n+2},\dots)\in \{0,1,\dots,m-1\}^\BN:\\ &x=\sum_{k=1}^\infty
\e_k\be^{-k} \Bigr\}
\end{align*}
and
\[
\mathcal N_n(x;\be)=\#\E_n(x;\be).
\]
(We will write simply $\mathcal N_n(x)$ if it is clear what $\be$ is under consideration.) In other words, $\mathcal N_n(x)$ counts the number of words of length~$n$ in the alphabet $\{0,1,\dots,m-1\}$ which can serve as prefixes of $\be$-expansions of $x$. We will be interested in the rate of growth of the function $x\mapsto\mathcal N_n(x)$.

Let $\be>1$ be a {\em Pisot number} (an algebraic integer whose conjugates are less than~1 in modulus). Our central result is the following

\begin{thm}
\label{thm-1.1}
There exists a constant $\gamma=\gamma(\be,m)>0$ such that
\begin{equation}
\label{e-1.2}
\lim_{n\to \infty} \frac{\log \mathcal N_n(x;\be)}{n}=\gamma\quad \mbox{ for $\L$-a.e. }\! x\in I_\be,
\end{equation}
where $\L$ denotes the Lebesgue measure.
\end{thm}

Let $\mu=\mu_{\beta,m}$ denote the probability measure on $\R$ defined as follows:
\[
\mu(E)=\mathbb P\left\{(\e_1,\e_2,\dots)\in\{0,1,\dots,m-1\}^\BN : \sum_{k=1}^\infty \e_k\be^{-k}\in E\right\},
\]
where $\mathbb P=\prod_1^\infty\{1/m,\dots,1/m\}$, and $E$ is an arbitrary Borel subset of $\BR$.

Recall that a Borel probability measure $\nu$ on $\R$ is called {\em self-similar} if $\nu=\sum_{i=1}^r p_i\,\nu\circ T_i^{-1}$, where $T_1,\dots,T_r$ are linear contractions on $\BR$, $p_i\ge0$ with $\sum_{i=1}^r p_i=1$.
The measure $\mu$ is known to be a self-similar measure supported on $I_\be$ with $r=m$, $T_ix=(x+i)/\be$ ($i=0,1,\dots,m-1$) and $p_i\equiv1/m$ (\cite{Hut81}). When $m=2$,  $\mu$ is the so-called {\it Bernoulli convolution associated with $\beta$} -- see, e.g., \cite{Sol}. For $x\in I_\be$, the {\it local dimension} of $\mu$ at $x$ is defined by
       \begin{equation}\label{e-local}
             d(\mu,x)= \lim_{r\to 0}\frac{\log \mu([x-r,x+r])}{\log r},
              \end{equation}
provided that the limit exists. As an application of Theorem~\ref{thm-1.1}, we obtain

\begin{cor}
\label{cor-1.2}
For $\L$-a.e.\! $x\in I_\be$,
$d(\mu_{\be,m},x)\equiv({\log m -\gamma})/{\log \beta}$.
\end{cor}

\begin{thm}
\label{thm-1.2}
If $\beta$ is an integer such that  $\beta$ divides $m$, then $\gamma=\log(m/\beta)$. Otherwise we have
$\gamma<\log(m/\beta)$.
\end{thm}

Theorem~\ref{thm-1.1}, Corollary~\ref{cor-1.2} and Theorem~\ref{thm-1.2} together yield

\begin{pro}\label{cor2}
We have $d(\mu_{\be,m},x)\equiv D_{\be,m}$ for Lebesgue-a.e. $x\in I_\be$ with $1\le D_{\be,m}<\log_\be m$. Moreover, $D_{\be,m}>1$ unless $\be$ is an integer dividing $m$.
\end{pro}

In addition to the above results for Pisot $\be$, we also obtain a general result for all small $\be$ which holds for all internal $x$. Recall that if $\be\in\bigl(1, \frac{1+\sqrt{5}}{2}\bigr)$ and $m=2$, then any $x\in \bigl(0,\frac{1}{\beta-1}\bigr)$ has a continuum of distinct $\be$-expansions (see \cite[Theorem~3]{EJK}). We prove a quantitative version of this claim for an arbitrary $m\ge2$:

\begin{thm}\label{thm-1.5}
Let $\beta$ be an arbitrary number in $\bigl(1, \frac{1+\sqrt{5}}{2}\bigr)$. Then there exists $\kappa=\kappa(\be)>0$ such that
\begin{equation}
\label{e-1.3}
\liminf_{n\to \infty} \frac{\log_2\mathcal N_n(x;\be)}{n}\geq \kappa\quad \mbox{ for any } x\in \Bigl(0,\frac{m-1}{\beta-1}\Bigr).
\end{equation}
\end{thm}

\begin{cor}\label{cor-1.5}
For any $\be\in\bigl(1, \frac{1+\sqrt{5}}{2}\bigr)$ and $m=2$, we have
\[
\overline d(\mu,x)\le(1-\kappa)\log_\beta 2
\]
for all $x\in\bigl(0,\frac{1}{\beta-1}\bigr)$, where
\[
\overline d(\mu,x)= \limsup_{r\to 0}\frac{\log \mu([x-r,x+r])}{\log r}.
\]
\end{cor}

The content of the paper is the following. In Section~2, we prove Theorem~\ref{thm-1.1} and Corollary~\ref{cor-1.2}.
Section~3 is devoted to the proof of Theorem~\ref{thm-1.2}. In Section~4, we consider an important class of examples, namely, the case when $\beta$ is a multinacci number. In Section~5, we prove Theorem~\ref{thm-1.5} and give an explicit lower bound for $\kappa$.

\section{Proof of Theorem \ref{thm-1.1} and Corollary \ref{cor-1.2}}
\setcounter{equation}{0}

First we reformulate our problem in the language of iterated function systems (IFS). Note that
\begin{equation}\label{eq:En}
\E_n(x;\be)=\left\{(\e_1,\dots, \e_n)\in\{0,1,\dots,m-1\}^n \mid 0\le
x-\sum_{k=1}^n \e_k\be^{-k} \le \frac{(m-1)\be^{-n}}{\be-1} \right\}
\end{equation}
(see, e.g., \cite{HS}). Consider now the following IFS $\Phi=\{S_i\}_{i=1}^{m}$ on $\R$:
\begin{equation}
\label{e-1.0}
S_i(x)=\rho x+ (i-1)(1-\rho)/(m-1),\qquad i=1,\ldots,m,
\end{equation}
where $\rho=1/\beta\in(0,1)$. Since $m>\be$, it is clear that $[0,1]$ is the attractor of $\Phi$ (note that $S_m(1)=1$), i.e., $[0,1]=\bigcup_{i=1}^m S_i([0,1])$.

Let $\A$ denote the alphabet $\{1,\ldots,m\}$ and  $\A_n$ the collection of all words of length $n$ over $\A$, $n\in \N$.
For $z\in I_\be$ it is clear that
\begin{equation}
\label{e-1.1}
\mathcal N_n\left(\frac{(m-1)z}{\be-1}\right)=\#\{J=j_1\cdots j_n\in \A_n:\; z\in S_{J}([0,1])\},
\end{equation}
where
$S_J:=S_{j_1}\circ S_{j_2}\circ \cdots\circ S_{j_n}$. This is because
\[
S_J(z)=\frac{1-\rho}{m-1}\sum_{k=1}^n (j_k-1)\rho^{k-1}+\rho^n z,
\]
and thus, $S_J([0,1])=\bigl[\frac{1-\rho}{m-1}\sum_{k=1}^n (j_k-1)\rho^{k-1}, \frac{1-\rho}{m-1}\sum_{k=1}^n (j_k-1)\rho^{k-1}+\rho^n\bigr]$, which is none other than a rescaled version of (\ref{eq:En}).

We sketch here the proof of Theorem~\ref{thm-1.1}:  first we  encode the interval $[0,1]$ as a cylinder in a subshift space of finite type, and show that  $\mathcal N_n(\frac{m-1}{\be-1}z)$ corresponds to the norm of a matrix product which depends on the coding of $z$ and $n$. Next, we construct an irreducible branch of the subshift in question and assign an invariant Markov measure such that its projection under the coding map is equivalent to the Lebesgue measure on a subinterval of $[0,1]$. Then by the subadditive ergodic theorem,  $\lim_{n\to \infty} \frac{\log \mathcal N_n(\frac{m-1}{\be-1}z)}{n}$ equals a non-negative constant $\L$-a.e.~on this subinterval; in the end we extend the result to the whole interval $[0,1]$.

Finally, we apply the theory of random $\beta$-expansions to show that this constant $\gamma$ is strictly positive.

\subsection{Coding of  $[0,1]$ and matrix products}
In this part, we will encode $[0,1]$ via a subshift and show that $\mathcal N_n(\frac{m-1}{\be-1}x)$ can be expressed in terms of  matrix products.
This approach mainly follows \cite{Fen05}.

For $n\in \N$, define
$$
P_n=\{S_J(0):\; J\in \A_n\}\cup \{S_J(1):\; J\in \A_n\}.
$$
The points in $P_n$, written as $h_1,\cdots, h_{s_n}$ (ranked in the increasing order),  partition $[0,1]$ into non-overlapping closed intervals which are called {\it $n$-th net intervals}. Let $\F_n$ denote the collection of $n$-th net intervals, that is,
$$
{\cal F}_n=\left\{[h_j,h_{j+1}]:\  j=1,\ldots, s_n-1\right\}.
$$
For convenience we write $\F_0=\{[0,1]\}$. Since  $P_n\subset P_{n+1}$, we obtain the following  net properties:
 \begin{itemize}
 \item[(i)]
  $\bigcup_{\Delta\in {\cal F}_n}\Delta=[0,1]$ for any $n\geq 0$;
\item[(ii)] For any $\Delta_1, \Delta_2\in {\cal F}_n$ with $\Delta_1\neq \Delta_2$, $\mbox{int}(\Delta_1)\cap
\mbox{int}(\Delta_2)=\emptyset$;
\item[(iii)] For any $\Delta\in {\cal F}_n$ ($n\geq 1$),
there is a unique element $\widehat{\Delta}\in {\cal F}_{n-1}$ such that    $\widehat{\Delta}\supset  \Delta$.
\end{itemize}

For $\Delta=[a,b]\in {\cal F}_n$, we define
 \begin{equation}
 \label{mul}
 \begin{split}
\mathcal N_n(\Delta)&=\#\left\{J\in {\cal A}_n:\ S_{J}\left((0,1)\right)\cap \Delta\neq \emptyset\right\}\\
&=\#\left\{J\in {\cal A}_n:\ S_{J}\left([0,1]\right)\supset  \Delta\right\}.
\end{split}
\end{equation}
It is easy to see that
\begin{equation}
\label{e-2.5}
\mathcal N_n\Bigl(\frac{m-1}{\be-1}z\Bigr)=\mathcal N_n(\Delta)\quad \mbox{ for any $\Delta\in \F_n$ and each  $z\in {\rm int}(\Delta)$},
\end{equation}
where $\mathcal N_n(z)$ is defined as in (\ref{e-1.1}).

As shown in \cite{Fen05}, the interval $[0,1]$ can be coded via a subshift of finite type, and for each $n\geq 1$ and $\Delta\in \F_n$,  $\mathcal N_n(\Delta)$ corresponds to the norm of certain matrix product which depends on the coding of $\Delta$.  More precisely, the following results (C1)-(C4) were obtained in \cite[Section 2]{Fen05}:

\begin{itemize}
\item[(C1)] There exist a finite alphabet $\Omega=\{1,\ldots, r\}$ with $r\geq 2$ and an $r\times r$  matrix $A=(A_{ij})$ with $0$-$1$ entries such that for each $n\geq 0$, there is a one-to-one surjective map $\phi_n:\; \F_n\to \Omega_{A,n+1}^{(1)}$, where
    $$
    \Omega_{A,n+1}^{(1)}=\left\{x_1\ldots x_{n+1}\in \Omega^{n+1}:\; x_1=1,\; A_{x_ix_{i+1}}=1\mbox{ for } 1\leq i\leq n\right\}.
    $$
    The map $\phi_n$ is called the {\it $n$-th coding map} and for $\Delta\in \F_n$,  $\phi_n(\Delta)$  is called {\it the $n$-th coding of $\Delta$}.

    \item[(C2)] The coding maps $\phi_n$ preserve the net structure in the sense that for any  $x_1\ldots x_{n+2}\in \Omega_{A,n+2}^{(1)}$,
    $$\phi_{n+1}^{-1}(x_1\ldots x_{n+2})\subseteq \phi_{n}^{-1}(x_1\ldots x_{n+1}).$$

\item[(C3)] There is a family of positive numbers $\ell_i$,  $1\leq i\leq r$, such that
for each $\Delta\in \F_n$ with $\phi_n(\Delta)=x_1\ldots x_{n+1}$,
$$
|\Delta|=\ell_{x_{n+1}} \rho^n,
$$
where $|\Delta|$ denotes the length of $\Delta$.
\item [(C4)] There are a family of positive integers $v_i$, $1\leq i\leq r$, with $v_1=1$,  and a family of non-negative matrices
$$
\left\{T(i,j):\; 1\leq i,j\leq r,\; A_{ij}=1\right\}
$$
with $T(i,j)$ being a $v_i\times v_j$ matrix, such that for each $n\geq 1$ and $\Delta\in \F_n$,
\begin{equation}
\label{e-m}
\mathcal N_n(\Delta)=\|T(x_1,x_2)\ldots T(x_n,x_{n+1})\|,
\end{equation}
where $x_1\ldots x_{n+1}=\phi_n(\Delta)$, $\|M\|$ denotes the sum of the absolute values of entries of $M$. Furthermore, the product  $T(x_1,x_2)\ldots T(x_n,x_{n+1})$ is a strictly positive $v_{x_{n+1}}$-dimensional row vector.

\end{itemize}

To prove Theorem \ref{thm-1.1}, we still need the  following property of $\Omega$, which was proved in
\cite[Lemma 6.4]{Fen07}):
\begin{itemize}
\item[(C5)] There is a  non-empty subset $\widehat{\Omega}$ of $\Omega$ satisfying the following properties:
 \begin{itemize}
 \item[(i)] $\{j\in
\Omega:\ A_{ij}=1\}\subseteq \widehat{\Omega}$ for any
$i\in \widehat{\Omega}$.
\item[(ii)]
For any $i,j\in
\widehat{\Omega}$, there exist $x_1,\ldots, x_n\in
\widehat{\Omega}$ such that $x_1=i$, $x_n=j$
and $A_{x_kx_{k+1}}=1$ for $1\leq k\leq n-1$.
\item[(iii)]
 For any
$i\in \Omega$ and $j\in \widehat{\Omega}$, there exist $x_1,\ldots, x_n\in
{\Omega}$  such that $x_1=i$, $x_n=j$
and $A_{x_kx_{k+1}}=1$ for $1\leq k\leq n-1$.
\end{itemize}
\end{itemize}

\begin{re}
\label{re-2.1}
Since $\F_n$ has the net structure, we have for each $\Delta\in \F_n$,
$$|\Delta|=\sum_{\Delta'\in F_{n+1},~\Delta'\subseteq \Delta}|\Delta'|,
$$
which together with (C1)-(C3) yields
\begin{equation}
\label{e-2.u}
\ell_i=\rho \sum_{j\in \Omega,~A_{ij}=1} \ell_j \quad \mbox{ for all  } i\in \Omega.
\end{equation}
By part (i) of (C5), we have also
\begin{equation}
\label{e-2.v}
\ell_i=\rho \sum_{j\in \widehat{\Omega},~A_{ij}=1} \ell_j \quad \mbox{ for all  } i\in \widehat{\Omega}.
\end{equation}

\end{re}

\subsection{Proof of Theorem \ref{thm-1.1}}

In this part we prove the following

\begin{thm}
\label{thm-2.1}
There exists a constant $\gamma\geq 0$ such that for  $\Delta\in \F_k$, if the $k$-th coding  $y_1\ldots y_{k+1}=\phi_k(\Delta)$ of $\Delta$ satisfies $y_{k+1}\in \widehat{\Omega}$, then
\begin{equation}
\label{e-2.2}
\lim_{n\to \infty} \frac{\log \mathcal N_n(\frac{m-1}{\be-1}x)}{n}=\gamma  \mbox{ for  $\L$-a.e.\! $x\in \Delta$}.
\end{equation}
 \end{thm}

Let us first show that Theorem~\ref{thm-2.1} implies Theorem~\ref{thm-1.1}. To see it, we say a net interval $\Delta$ is  {\it good } if it satisfies the condition of Theorem~\ref{thm-2.1}. According to part~(iii) of (C5), there is an positive integer $N$ such that for any net interval  $\Delta\in  \F_n$, there is $k\leq N$ and an $(n+k)$-th net interval which is contained in $\Delta$ and is good. Hence by Theorem~\ref{thm-2.1} and (C2)-(C3), there is a constant $c>0$  such that for any net interval~$\Delta$, (\ref{e-2.2}) holds for a sub-net-interval  of $\Delta$ with Lebesgue measure greater than $c|\Delta|$.   A recursive argument then shows that (\ref{e-2.2}) holds for $[0,1]$.

\begin{proof}[Proof of Theorem \ref{thm-2.1}]
 Consider the one-sided subshift of finite type $(\widehat{\Omega}^\N_A,~\sigma)$, where
$$
\widehat{\Omega}^\N_A=\left\{(x_i)_{i=1}^\infty:~ x_i\in \widehat{\Omega},~A_{x_ix_{i+1}}=1\mbox{ for }  i\geq 1\right\},
$$
and $\sigma$ is the left shift defined by $(x_i)_{i=1}^\infty\mapsto (x_{i+1})_{i=1}^\infty$. By parts (i)-(ii) of (C5),
$(\widehat{\Omega}^\N_A,~\sigma)$ is topologically transitive. Define a matrix $P=(P_{ij})_{i,j\in \widehat{\Omega}}$ by
\begin{equation}
\label{e-t}
P_{ij}=\left\{
\begin{array}{ll}
\rho\ell_j/\ell_i & \mbox{ if } A_{ij}=1,\\
0 &\mbox{ otherwise}.
\end{array}
\right.
\end{equation}
By (\ref{e-2.v}) and part (ii) of (C5), $P$ is an irreducible transition matrix. Hence there is a unique $\#(\widehat{\Omega})$-dimensional positive probability vector ${\bf p}=(p_i)_{i\in \widehat{\Omega}}$ so that ${\bf p} P={\bf p}$. Let $\eta$ be the $({\bf p}, P)$-Markov measure on $\widehat{\Omega}^\N_A$, i.e.,
$$
\eta([x_1\ldots x_n])=p_{x_1}P_{x_1x_2}\ldots P_{x_{n-1}x_n}
$$
for any cylinder set $[x_1\ldots x_n]$ in $\widehat{\Omega}^\N_A$. Since $P$ is irreducible, $\eta$ is ergodic\footnote{The reader may actually check that $\eta$ is the unique invariant measure of  maximal entropy, the so-called {\it Parry measure} on  $\widehat{\Omega}^\N_A$ -- see, e.g., \cite{Wa} for the definition.}. By the definition of $P$, we can check that
\begin{equation}
\label{e-2.11}
\eta([x_1\ldots x_n])=p_{x_1}\ell_{x_n}\rho^{n-1}
\end{equation}
for any cylinder set $[x_1\ldots x_n]$ in $\widehat{\Omega}^\N_A$.

Consider the family of matrices $\{ T(i,j):~ i,j\in \widehat{\Omega},~A_{i,j}=1\}$. Observe that for any $x_1\ldots x_{n+m}\in \widehat{\Omega}_{A,n+m}$,
\begin{eqnarray*}
&\mbox{}&\|T(x_1,x_2)\ldots T(x_{n+m-1},x_{n+m})\|\\
&\mbox{}&\quad ={\bf e}_{v_{x_1}}T(x_1,x_2)\ldots T(x_{n+m-1},x_{n+m}){\bf e}_{v_{x_{n+m}}}^t\\
&\mbox{}&\quad \leq {\bf e}_{v_{x_1}}T(x_1,x_2)\ldots T(x_{n-1},x_{n}){\bf e}_{v_{x_{n}}}^t
{\bf e}_{v_{x_n}} T(x_n,x_{n+1} )\ldots  T(x_{n+m-1},x_{n+m}){\bf e}_{v_{x_{n+m}}}^t\\
&\mbox{}& \quad  =\|T(x_1,x_2)\ldots T(x_{n-1},x_{n})\|\cdot\|T(x_n,x_{n+1})\ldots T(x_{n+m-1},x_{n+m})\|,
\end{eqnarray*}
where ${\bf e}_k$ denotes the $k$-dimensional row vector $(1,1,\ldots,1)$, and ${\bf e}_k^t$ denotes the transpose of ${\bf e}_k$.
By the Kingman subadditive ergodic theorem, there exists a constant $\gamma\geq 0$ such that
\begin{equation}
\label{e-er}
\lim_{n\to \infty}\frac{1}{n}\log \|T(x_1,x_2)\ldots T(x_{n-1},x_n)\|=\gamma\quad \mbox{ for $\eta$-a.e.\! $x=(x_i)_{i=1}^\infty\in \widehat{\Omega}^\N_A$}.
\end{equation}

Now assume that $\Delta$ is a $k$-th net interval with the coding $\phi_k(\Delta)=y_1\ldots y_{k+1}$ such that $y_{k+1}\in \widehat{\Omega}$.
Define the projection map $\pi:~[y_{k+1}]\to \R$ by
\begin{equation}
\{\pi(x)\}=\bigcap_{n=1}^\infty \phi_{n+k}^{-1}(y_1\ldots y_k x_1\ldots x_{n+1}),\quad x=(x_i)_{i=1}^\infty\in \widehat{\Omega} \mbox{ with }x_1=y_{k+1}.
\end{equation}
Since the coding maps preserve the net structure (see (C2)), the projection $\pi$ is well defined and is one-to-one,  except for a countable set on which it is is two-to-one.
Let $\nu=\eta|_{[y_{k+1}]}$ be the restriction of $\eta$ on the cylinder $[y_{k+1}]$. Let $\nu\circ \pi^{-1}$ be the projection of $\nu$ under $\pi$.

We claim that $\nu\circ \pi^{-1}$ is equivalent to ${\mathcal L}|_{\Delta}$, the Lebesgue measure restricted on $\Delta$, in the sense that there exists a constant $C\geq 1$ such that $C^{-1}{\mathcal L}|_{\Delta}\leq \nu\circ \pi^{-1}\leq C {\mathcal L}|_{\Delta}$.
The claim just follows from the fact that for each sub net interval $\Delta'$ with coding $y_1\ldots y_k x_1\ldots x_{n+1}$,
   $$\nu\circ \pi^{-1}(\Delta')=\eta([x_1\ldots x_{n+1}])=p_{x_1}\ell_{x_{n+1}}\rho^{n}=p_{y_{k+1}}\rho^{-k} |\Delta'|,$$
where we use (\ref{e-2.11}) and (C3). Since the collection of sub net intervals of $\Delta$ generates the Borel sigma-algebra on $\Delta$, $\nu\circ \pi^{-1}$ only differs from ${\mathcal L}|_{\Delta}$ by a constant. The claim thus follows.

Now assume that $x=(x_i)_{i=1}^\infty\in [y_{k+1}]$ such that $z=\pi(x)\not\in \bigcup_{n\geq 0}P_n$.  Then by (\ref{e-2.5}),
\begin{eqnarray*}
\mathcal N_{n+k}\Big(\frac{m-1}{\be-1}z\Big)&=&\|T(y_1,y_2)\ldots T(y_k,y_{k+1})T(x_1,x_2)\ldots T(x_n,x_{n+1})\|\\
&\asymp& \|T(x_1,x_2)\ldots T(x_n,x_{n+1})\|,
\end{eqnarray*}
where we use the fact that $T(y_1,y_2)\ldots T(y_k,y_{k+1})$ is a strictly positive vector (see (C4)), and  the notation $a_n\asymp b_n$ means that   $C^{-1}b_n \leq a_n\leq Cb_n$ for  a positive constant $C\geq 1$ independent of $n$. This together with (\ref{e-er}) yields
$$
\lim_{n\to \infty}\frac{1}{n}\log \mathcal N_n\Big(\frac{m-1}{\be-1}\pi(x)\Big)=\gamma \quad\mbox{ for $\eta$-a.e.\! $x\in [y_{k+1}]$},
$$
and hence
$$
\lim_{n\to \infty}\frac{1}{n}\log \mathcal N_n(z)=\gamma \quad\mbox{ for $\nu\circ \pi^{-1}$-a.e.\! $z\in \R$}.
$$
Since $\nu\circ \pi^{-1}$ is equivalent to ${\mathcal L}|_\Delta$, we obtain Theorem \ref{thm-2.1} (and thus, Theorem~\ref{thm-1.1}) with $\gamma\ge0$.
\end{proof}

\subsection{Proof that $\gamma>0$.}\label{subsection2.3} Let us consider first the case of non-integer $\be$. It is clearly sufficient to prove $\ga>0$ for $m=\lfloor\beta\rfloor+1$. Following \cite{DdV}, we introduce the random $\be$-transformation $K_\be$. Namely, put
\begin{equation}\label{eq:Sk}
S_k=\left[\frac k{\be},\frac{\lfloor\beta\rfloor}{\be(\be-1)}+\frac{k-1}{\be} \right]
\end{equation}
(the switch regions) and
\[
E_k=\left(\frac{\lfloor\beta\rfloor}{\be(\be-1)}+ \frac{k-1}{\be},\frac{k+1}{\be}\right),\quad k=1,\dots,\lfloor\beta\rfloor-1,
\]
with
\[
E_0=\bigl[0,\frac1{\be}\bigr),\quad E_{\lfloor\beta\rfloor}=\left(\frac{\lfloor\beta\rfloor} {\be(\be-1)}+ \frac{\lfloor\beta\rfloor-1}{\be}, \frac{\lfloor\beta\rfloor}{\be-1}\right]
\]
(the equality regions). Put now $\Om=\{0,1\}^\BN$ and the map $K_\be:\Om\times I_\be \to \Om\times I_\be$ defined as
\[
K_\be(\om,x)=\begin{cases}
(\omega,\be x-k), & x\in E_k,\ k=0,1,\dots, \lfloor\beta\rfloor,\\
(\sigma(\omega),\be x-k), & x\in S_k\ \text{and}\ \om_1=1,\ k=1,\dots,\lfloor\beta\rfloor,\\
(\sigma(\omega),\be x-k+1), & x\in S_k\ \text{and}\ \om_1=0,\ k=1,\dots,\lfloor\beta\rfloor,
\end{cases}
\]
where $\sigma(\omega_1,\omega_2,\omega_3,\dots)= (\omega_2,\omega_3,\dots)$. The map $K_\be$ generates all $\be$-expansions of $x$ by acting as a shift -- see \cite[p.~159]{DdV} for more details. More precisely, if $x\in E_k$, then the first digit of its $\be$-expansion must be $k$; if $x\in S_k$, it can be either $k$ or $k-1$.

It was shown in \cite{DdV} that there exists a unique probability measure $m_\be$ on $I_\be$ such that $m_\be$ is equivalent to the Lebesgue measure and $\mathbb P\otimes m_\be$ is invariant and ergodic under $K_\be$, where $\mathbb P=\prod_1^\infty \left\{\frac12,\frac12\right\}$.\footnote{It should also be noted that $K_\be$ has a unique ergodic measure of maximal entropy which is singular with respect to $\mathbb P \otimes m_\be$ and whose projection onto the second
coordinate is precisely $\mu_{\be,m}$ -- see \cite{DdV2} for more detail.}

The famous Garsia separation lemma (\cite[Lemma~1.51]{Ga0}) states that there exists a constant $C=C(\be,m)>0$ such that if $\sum_{j=1}^n \e_j\be^{-j}\neq \sum_{j=1}^n \e'_j\be^{-j}$ for some $\e_j,\e'_j\in\{0,1,\dots,m-1\}$, then $\left|\sum_{j=1}^n (\e_j-\e'_j)\be^{-j}\right|\ge C\be^{-n}$. Hence
\begin{equation}\label{eq:garsia}
\#\left\{\sum_{j=1}^n \e_j\be^{-j} \mid \e_j\in\{0,1,\dots,m-1\}\right\}=O(\beta^n).
\end{equation}
In particular, there exist $k\ge 2$ and two words $a_1\dots a_k$ and $b_1\dots b_k$ with $a_j,b_j\in\{0,1,\dots,\lfloor\be\rfloor\}$ such that $\sum_{j=1}^k a_j\be^{-j}= \sum_{j=1}^k b_j\be^{-j}$.

Let $J_{a_1\dots a_k}$ denote the interval of $x$ which can have $a_1\dots a_k$ as a prefix of their $\be$-expansions. (It is obvious that $J_{a_1\dots a_k}=\bigl[\sum_1^k a_j\be^{-j}, \sum_1^k a_j\be^{-j}+\frac{\lfloor\be\rfloor}{\be-1}\be^{-k}\bigr]$.) It follows from the ergodicity of $K_\be$ and \cite[Lemma~8]{DdV} that for $\mathbb P\otimes\L$-a.e. $(\om,x)\in\Omega\times I_\be$ the block $a_1\dots a_k$ appears in the $\be$-expansion of $x$ (specified by $\omega$) with a limiting frequency $\widetilde\gamma>0$.

In particular, for $\mathcal L$-a.e. $x$ there exists a $\be$-expansion $(\e_1,\e_2,\dots)$ which contains the block $a_1\dots a_k$ with the positive limiting frequency $\widetilde\gamma$, i.e.,
\[
\lim_{n\to\infty}\frac1n\ \#\{j: \e_j\dots\e_{j+k-1}=a_1\dots a_k\}=\widetilde\gamma.
\]
Since any such block can be replaced with $b_1\dots b_k$, and the resulting sequence remains a $\be$-expansion of $x$, we conclude, in view of (\ref{e-1.2}), that $\gamma/\log 2\ge\widetilde\gamma>0$.

Let now $\be\in\BN$, so $m\ge\be+1$. In a $\be$-expansion with digits $\{0,1,\dots,m-1\}$ one can replace the block $10$ with $0\be$ without altering the rest of the expansion. Since for $\mathcal L$-a.e. $x$ its $\be$-ary expansion (with digits $0,1,\dots,\be-1$) contains the block $01$ with the limiting frequency $\be^{-2}>0$, we conclude that $\gamma/\log 2\ge \be^{-2}>0$.

The proof of Theorem~\ref{thm-1.1} is complete.

The same argument as above proves
\begin{pro}
If $\be$ satisfies an algebraic equation with integer coefficients bounded by $m$ in modulus, then there exists $C=C(\be,m)>0$ such that
\begin{equation}\label{eq:ineq2}
\liminf_{n\to \infty} \frac{\log \mathcal N_n(x;\be)}{n}\ge C \mbox{ for $\L$-a.e. }\! x\in I_\be.
\end{equation}
\end{pro}
It is an intriguing open question whether (\ref{eq:ineq2}) holds for all $\be>1$. (See also Section~\ref{sec:gr}.)

\subsection{Proof of Corollary~\ref{cor-1.2}}
Note first that (\ref{eq:En}) can be rewritten as follows:
\[
\E_n(x;\be)=\left\{(\e_1,\dots,\e_n)\in\{0,1,\dots,m-1\}^n: x-\frac{(m-1)\be^{-n}}{\be-1}\le \sum_{k=1}^n \e_k\be^{-k}\le x\right\}.
\]
Thus, if $(\e_1,\dots,\e_n)\in\E_n(x;\be)$, then for any $(\e_{n+1},\e_{n+2},\dots)\in\{0,1,\dots,m-1\}^\N$ we have
\[
x-\frac{(m-1)\be^{-n}}{\be-1}\le \sum_{k=1}^\infty \e_k\be^{-k}\le x+\frac{(m-1)\be^{-n}}{\be-1}.
\]
Hence by definition,
\begin{equation}\label{eq:ineq1}
\mu\left(x-\frac{(m-1)\be^{-n}}{\be-1}, x+\frac{(m-1)\be^{-n}}{\be-1}\right)\ge m^{-n}\mathcal N_n(x;\be).
\end{equation}
Put now
\begin{align*}
\E'_n(x;\be)=\Bigl\{&(\e_1,\dots,\e_n)\in\{0,1,\dots,m-1\}^n:
\\ &x-\frac{(m-1)\be^{-n}}{\be-1}-\frac{\be^{-n}}{n^2}\le \sum_{k=1}^n \e_k\be^{-k}\le x+\frac{\be^{-n}}{n^2}\Bigr\}.
\end{align*}
We are going to need the following
\begin{lem}\label{lem:finite}
For $\mathcal L$-a.e. $x\in I_\be$ we have $\E'_n(x;\be)=\E_n(x;\be)$ for all $n$, except, possibly, a finite number (depending on $x$).
\end{lem}
\begin{proof}We have
\begin{align*}
\E'_n(x;\be)\setminus\E_n(x;\be)&=\left\{(\e_1,\dots,\e_n): 0< x-\frac{(m-1)\be^{-n}}{\be-1}- \sum_{k=1}^n \e_k\be^{-k}\le \frac{\be^{-n}}{n^2}\right\} \\
 &\cup \left\{(\e_1,\dots,\e_n): 0< \sum_{k=1}^n\e_k\be^{-k}-x\le\frac{\be^{-n}}{n^2}\right\}.
\end{align*}
Hence, in view of (\ref{eq:garsia}),
\[
\L\left\{x : \E'_n(x;\be)\setminus\E_n(x;\be)\neq\emptyset \right\}=O\left(\frac{1}{n^2}\right),
\]
whence by the Borel-Cantelli lemma,
\[
\L\bigl\{x : \E'_n(x;\be)\setminus\E_n(x;\be)\neq\emptyset \,\,\text{for an infinite set of $n$}\bigr\}=0.
\]
\end{proof}

Return to the proof of the corollary. Put
\[
\mathcal D'_n(x;\be)=\left\{(\e_1,\e_2,\dots) : x-\frac{\be^{-n}}{n^2}\le \sum_{k=1}^\infty\e_k\be^{-k}\le x+\frac{\be^{-n}}{n^2}\right\}.
\]
Note that if $(\e_1,\e_2,\dots)\in\mathcal D'_n(x;\be)$, then $(\e_1,\dots,\e_n)\in\mathcal E'_n(x;\be)$, since $\sum_1^n \e_k\be^{-k}\ge\sum_1^\infty\e_k\be^{-k}- \frac{(m-1)\be^{-n}}{\be-1}$. Thus, by Lemma~\ref{lem:finite}, for $\mathcal L$-a.e. $x$ and all sufficiently large $n$,
\[
\mu\left(x-\frac{\be^{-n}}{n^2},x+\frac{\be^{-n}}{n^2}\right) \le m^{-n}\mathcal N_n(x;\be).
\]
Together with (\ref{eq:ineq1}), we obtain for $\mathcal L$-a.e. $x$,
\[
\mu\left(x-\frac{\be^{-n}}{n^2},x+\frac{\be^{-n}}{n^2}\right) \le m^{-n}\mathcal N_n(x;\be)\le \mu\left(x-\frac{(m-1)\be^{-n}}{\be-1}, x+\frac{(m-1)\be^{-n}}{\be-1}\right).
\]
Taking logs, dividing by $n$ and passing to the limit as $n\to\infty$ yields the claim of Corollary~\ref{cor-1.2}.\footnote{It is easy to see that $d(\mu,x)$ exists if the limit in (\ref{e-local}) exists along some exponentially decreasing subsequence of $r$.}

\section{Proof of Theorem \ref{thm-1.2}}
\setcounter{equation}{0}
We first introduce some notation.  For $q\in \R$, we use $\tau(q)$
to denote the {\it  $L^{q}$ spectrum of $\mu$}, which is defined
 by
\begin{equation*}
{\tau}(q)=\liminf_{r\rightarrow 0+}\frac{\log \left( \sup
\sum_{i}\mu ([x_{i}-r,x_i+r])^{q}\right) }{\log r},
\end{equation*}
where the supremum is taken over all the disjoint families $%
\{[x_{i}-r,x_i+r]\}_{i}$ of closed intervals with $x_{i}\in [0,1]$.
It is easily checked that ${\tau}(q)$ is a concave
function of $q$ over $\R$,   $\tau(1)=0$ and $\tau(0)=-1$. For $\alpha\geq 0$, let
$$E(\alpha)=\{x\in [0,1]:\; d(\mu,x)=\alpha\},$$ where $d(\mu,x)$ is defined as in (\ref{e-local}).  The following lemma is a basic fact in multifractal analysis (see, e.g., \cite[Theorem 4.1]{LaNg99} for a proof).

\begin{lem} Let $\alpha\geq 0$.
If $E(\alpha)\neq \emptyset$, then
\begin{equation}
\label{e-1}
\dim_H E(\alpha)\leq \alpha q-\tau(q), \qquad \forall \;q\in \R.
\end{equation}
\end{lem}
\medskip

\begin{proof}[Proof of Theorem~\ref{thm-1.2}]
Set $t=({\log m -\gamma})/{\log \beta}$. By Corollary~\ref{cor-1.2}, we have $d(\mu,x)=t$ for $\L$-a.e. \!$x\in [0,1]$.
It was proved in \cite[Proposition~5.3]{DST99} that $\mu$ is absolutely continuous  if and only if $\beta$ is an integer so that $\beta|m$. When $\mu$ is absolutely continuous, $d(\mu,x)=1$ for $\L$-a.e. \!$x\in [0,1]$ and hence
$t=1$, which implies that  $\gamma=\log(m/\beta)$.

In the following we assume that $\mu$ is singular. It was proved in \cite{PrUr89} that $\dim_H\mu<1$.
Since $d(\mu,x)=t$ for $\L$-a.e.\! $x\in [0,1]$, we have $\L(E(t))=1$ and hence $\dim_H E(t)=1$. By (\ref{e-1}), we have
    \begin{equation}
    \label{e-2}
    1\leq t q-\tau(q),\qquad \forall\; q\in \R.
    \end{equation}
    Taking $q=1$ in (\ref{e-2}) and using the fact $\tau(1)=0$, we have  $t\geq 1$.
     It was proved in \cite{Fen03} that $\tau(q)$ is differentiable for $q>0$ and $\dim_H\mu=\tau^\prime(1)$. Since $\tau$ is also concave, $\tau^\prime$ is continuous on $(0,+\infty)$.
         By (\ref{e-2})
     and the fact $\tau(0)=-1$, we have $\tau(q)-\tau(0)\leq t q$ for all $q\in \R$, which implies
    \begin{equation}
    \label{e-3.13}
    \tau^\prime(0+)\leq  t\leq \tau^\prime(0-).
    \end{equation}
     Since $\tau$ is concave, it is absolutely continuous on $[0,1]$ and hence
          \begin{equation}
          \label{e-3.14}
     1=\tau(1)-\tau(0)=\int_{[0,1]}\tau^\prime(x)\; dx.
     \end{equation}
          Since $\tau^\prime(1)=\dim_H\mu<1$, and $\tau^\prime$ is non-increasing on $(0,1)$, by (\ref{e-3.14}) we must have
           $\tau^\prime(0+)=\lim_{q\to 0+}\tau^\prime(q)>1$.
   This together with (\ref{e-3.13}) yields  $t>1$. Hence we have
   $\gamma<\log(m/\beta)$.
    \end{proof}

\begin{re}
 \begin{itemize}
  \item[(1)] It is interesting to compare Proposition~\ref{cor2} with a similar result for a Bernoulli-generic $x$. Let, for simplicity, $m=2$; then it is known that $d(\mu_{\be,2})\equiv H_\be<1$ for $\mu_\be$-a.e. $x$. -- see \cite{Lalley}. Here $H_\be$ is Garsia's entropy introduced in \cite{Ga} (see also \cite{HS} for some lower bounds for $H_\be$).
 \item[(2)] It was proved in \cite{Fen03} that the set of local dimensions of  $\mu$ contains the set
     $\{\tau^\prime(q):\; q>0\}$. In the  case that $\mu$ is singular, this set contains a neighborhood of $1$. To see it, just note that  $\tau^\prime(1)=\dim_H\mu<1<\tau^\prime(0+)$.

   \item[(3)]  We do not know whether the set of local dimensions of  $\mu$,
    $$
    \{\alpha\geq 0:\;E(\alpha)\neq \emptyset\},$$
     is always a closed interval. Nevertheless,  it was proved  in \cite{Fen07} that for each Pisot number $\beta$ and positive integer $m$, there exists an interval $I$ with $\mu(I)>0$  such that  the set of local dimensions of  $\mu|_I$ is always a closed interval, where $\mu|_I$ denotes the restriction of $\mu$ on $I$.
\item[(4)]  We conjecture  that  $\tau^\prime(0)$ exists. If this is true, by (\ref{e-3.13}) we have $t=\tau^\prime(0)$.
 \item[(5)] The following result can be proved in a way similar to the proof of Theorem \ref{thm-1.2}: assume that $\eta$ is a compactly supported Borel probability measure on $\R^d$ so that
 $d(\eta, x)=t$ on a set of Hausdorff dimension $d$. Then $t>d$ if $\tau'(1-)<d$. The reader is referred to \cite{Fen03} for the definitions of $d(\eta,x)$ and $\tau(q)$ for a measure  on $\R^d$.

    \end{itemize}
     \end{re}

\section{Examples}

As we have seen from the proof of Theorem~\ref{thm-1.1}, the exponent $\gamma$ in (\ref{e-1.2}) corresponds to the Lyapunov exponent of certain family of non-negative matrices. In the case when this family contains a rank-one matrix (for instance, this occurs when $v_i=1$ for some $i\in \widehat{\Omega}$), the corresponding matrix product is degenerate and one may obtain an explicit theoretic formula (via series expansion) for $\gamma$. Let us consider an important family of examples.

\begin{ex}
{\rm
Fix an integer $n\geq 2$. Let $\beta_n$ be the positive root of $x^n=x^{n-1}+\ldots+x+1$ (often called the $n$'th {\em multinacci number}). Let $m=2$. The following formula for $\gamma_n=\gamma(\be_n)$ was obtained  in \cite[Theorem~1.2]{Fen05}:
\begin{equation}
\label{e-4.1}
\gamma_n=\frac{\beta^{-n}\left(1-2\beta^{-n}\right)^2} {2-(n+1)\beta^{-n}}\sum\limits_{k=0}^\infty \left(\beta^{-nk}\sum_{J\in {\cal A}_k}\log
\|M_J\|\right),
\end{equation}
where $\A_0=\{\emptyset\}$ and $\A_k=\{1,2\}^k$ for $k\geq 1$. $M_\emptyset$ denotes the $2\times 2$ identity matrix, and   $M_1, M_2$ are two $2\times 2$ matrices given by
\begin{equation*}
M_1=\left[
\begin{array}{ll}
1 & 1 \\
0 & 1
\end{array}
\right] ,\;\;\;M_2=\left[
\begin{array}{ll}
1 & 0 \\
1 & 1
\end{array}
\right] .  \label{mm}
\end{equation*}
For $J=j_1\ldots j_k\in \A_k$,  $M_J$ denotes $M_{j_1}M_{j_2}\ldots
M_{j_n}$. For any $2\times 2$ non-negative matrix $B$,  $\|B\|=(1,1)B(1,1)^{t}.$

The numerical estimations in Table~\ref{tat} were given in \cite{Fen05} for $\gamma_n/\log 2$, $n=2,\ldots,10$. We also include in the table the approximate values for $D_{\be}=D_{\be,2}$ (see Corollary~\ref{cor-1.2} and Proposition~\ref{cor2}) and for Garsia's entropy $H_{\be}$ for comparison (taken from \cite{GKT}).

\begin{table}[tph]
\begin{tabular}{|c|c|c|c|c|}
\hline
$n$ & $\be_n$ &  $\gamma_n/\log 2$ & $D_{\be_n}$ & $H_{\be_n}$
\\
\hline
$2$ & $1.618034$ & $0.302\pm 0.001$ & $1.0054\pm 0.0015$ & $0.995713$\\ \hline
$3$ & $1.839287$ & $0.102500$ & 1.028876 & $0.980409$ \\ \hline
$4$ & $1.927562$ & $0.041560$ & 1.012318 & $0.986926$\\ \hline
$5$ & $1.965948$ & $0.018426$ & 1.006510 & $0.992585$\\ \hline
$6$ & $1.983583$ & $0.008590$ & 1.003341 & $0.996033$ \\ \hline
$7$ & $1.991964$& $0.004123$ & 1.001695 & $0.997937$ \\ \hline
$8$ & $1.996031$ & $0.002014$ & 1.000854 & $0.998945$\\ \hline
$9$ & $1.998029$ & $0.000993$ & 1.000429 & $0.999465$\\ \hline
$10$ & $1.999019$ & $0.000493$ & 1.000215 & $0.999731$\\ \hline
\end{tabular}
\bigskip
\caption{Approximate values of $\ga, D_\be$ and $H_\be$ for the multinacci family}
\label{tat}
\end{table}
}

\end{ex}

\section{Proof of Theorem~\ref{thm-1.5} and Corollary~\ref{cor-1.5}}
\label{sec:gr}


Let us first observe that without loss of generality we may confine ourselves to the case $m=2$. Indeed, if $m\ge3$ and $x\in\bigl(0,\frac1{\be-1}\bigr)$, then we can use digits $0,1$ and apply Theorem~\ref{thm-1.5} for $m=2$. If $x\in\bigl(\frac{j}{\be-1},\frac{j+1}{\be-1}\bigr)$ for $1\le j\le m-2$, then we put $y=x-\frac{j}{\beta-1}$ and apply Theorem~\ref{thm-1.5} for $m=2$ to $y$. For the original $x$ the claim will then follow with $\e_n\in\{j,j+1\}$.

If $x=\frac{j}{\be-1}$ with $1\le j\le m-2$, then we set $\e_1=j-1$ so
\[
\be\left(x-\frac{\e_1}{\be}\right)= \frac{\be+j-1}{\be-1}=\frac{\e_2}{\be}+\frac{\e_3}{\be^2} +\cdots
\]
It suffices to observe that $\frac{j}{\be-1}<\frac{\be+j-1}{\be-1}<\frac{j+1}{\be-1}$ and apply the above argument to $(\e_2,\e_3,\dots)$.

So, let $m=2$ and let $x\in I_\be$ have at least two $\be$-expansions; then there exists the smallest $n\ge0$ such that
$x\sim(\e_1,\dots,\e_n,\e_{n+1},\dots)_\be$ and
$x\sim(\e_1,\dots,\e_n,\e'_{n+1},\dots)_\be$ with $\e_{n+1}\neq \e'_{n+1}$.
We may depict this ``bifurcation'' as is shown in
Figure~\ref{fig:branching}.

\begin{figure}[t]
    \centering 
\centerline{
\begin{picture}(265,210)
\thicklines  \put(0,100){\line(1,0){100}}
\put(100,100){\line(0,-1){50}}
\put(100,100){\line(0,1){50}}\put(100,150){\line(1,0){120}}
\put(100,50){\line(1,0){100}}
\put(220,150){\line(0,1){30}} \put(220,150){\line(0,-1){30}}
\put(200,50){\line(0,1){30}} \put(200,50){\line(0,-1){30}}
\put(200,20){\line(1,0){50}}\put(200,80){\line(1,0){50}}
\put(220,120){\line(1,0){50}}\put(220,180){\line(1,0){50}}
\put(274,179){$\dots$}\put(274,119){$\dots$}
\put(254,79){$\dots$}\put(254,19){$\dots$}
\put(0,106){$\e_1\,\,\,\,\, \e_2\,\,\dots\,\, \e_{n-1}\,\,\, \e_n$}
\put(103,155){$\e_{n+1}\,\,\,\e_{n+2}\,\,\,\dots\,\,\dots\,\,
\e_{n_2}$}
\put(103,58){$\e'_{n+1}\,\,\,\e'_{n+2}\,\,\,\dots\,\,\,\e'_{n_2'}$}
\put(-10,97){$x$}
\end{picture}}
\caption{Branching and ``bifurcations''}
    \label{fig:branching}
\end{figure}
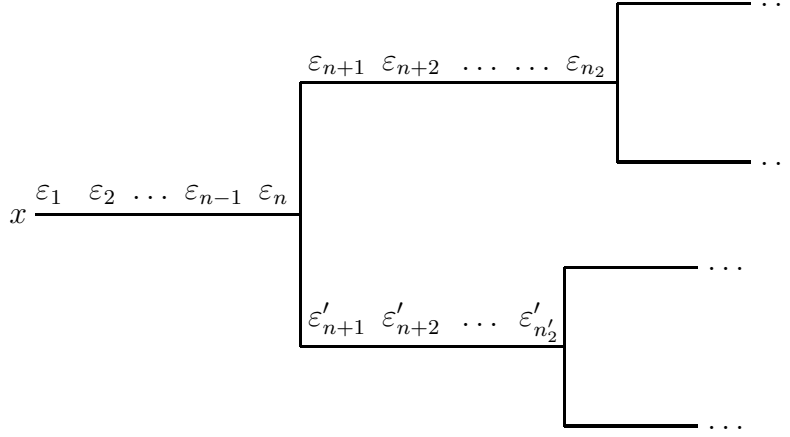

If $(\e_{n+1},\e_{n+2},\dots)$ is not a unique expansion, then there exists $n_2>n$ with the same property, etc. As a result, we obtain a subtree of the binary tree which corresponds to the set of all $\be$-expansions of $x$, which we call the {\it branching tree of $x$} and denote by $\T(x;\be)$.

The following claim is straightforward:

\begin{lem} \label{lem:estimate}
Suppose $K\in\N$ is such that the length of each branch is at most $K$. Then
\[
\mathcal N_n(x;\be) \ge 2^{n/K-1}.
\]
\end{lem}
\begin{proof}It is obvious that $\mathcal N_{Kn}(x;\be) \ge 2^{n}$, which yields the claim.
\end{proof}

\begin{thm}\label{thm:lessgr}
Suppose $1<\be<\frac{1+\sqrt5}2$ and put
\begin{equation}\label{eq:ga}
\kappa=\begin{cases} \frac12 \left(\left\lfloor \log_\be
\frac{\be^2-1}{1+\be-\be^2}\right\rfloor+1\right)^{-1},& \be>\sqrt2 \\
\frac12 \left(\left\lfloor \log_\be
\frac{1}{\be-1}\right\rfloor+1\right)^{-1}, & \be\le\sqrt2.
\end{cases}.
\end{equation}
Then for any $x\in\bigl[\frac1{\be},\frac1{\be(\be-1)}\bigr]$ we have
\begin{equation}\label{eq:Nnl}
\mathcal N_n(x;\be) \ge 2^{\kappa n-1}.
\end{equation}
\end{thm}
\begin{proof} In view of Lemma~\ref{lem:estimate}, it suffices to construct a subtree of $\T(x;\be)$ such that the length of its every branch is at most $1/\kappa$. Put $\Delta_\be=\bigl[\frac1{\be},\frac1{\be(\be-1)}\bigr]$; it is easy to check that one can choose different $\e_1$ for $x$ if and only if $x\in\Delta_\be$\footnote{Notice that $\De_\be$ is none other than $S_1$ given by (\ref{eq:Sk}) for $m=2$ and $\lfloor\be\rfloor=1$.}.

Put
\begin{equation}\label{eq:delta}
\delta=\min\,\left(\frac{1+\be-\be^2}{\be^2-1},\be-1\right)=
\begin{cases} \frac{1+\be-\be^2}{\be^2-1} & \text{if}\ \ \be>\sqrt2,\\
\be-1 & \text{if}\ \ \be\le\sqrt2
\end{cases}.
\end{equation}
Note that $\delta>0$, in view of $1<\be<\frac{1+\sqrt5}2$. Put $L_\be(x)=\be x,\ R_\be(x)=\be x-1$. The maps $L_\be$ and $R_\be$ act as shifts on the $\be$-expansions of $x$, namely, $L_\be(x)$ shifts a $\be$-expansion of $x$ if $\e_1=0$ and $R_\be$ -- if $\e_1=1$. Thus, by applying all possible compositions of the two maps we obtain all $\be$-expansions of $x$. (See subsection~\ref{subsection2.3} for more detail.)

Assume first that $x\in\De_\be$. We have two cases.

\medskip\noindent \textbf{Case 1.}
$x\in\bigl[\frac{1+\delta}{\be},
\frac1{\be(\be-1)}-\frac{\delta}{\be}\bigr]$. It is easy to check, using (\ref{eq:delta}), that this is an interval of positive length. Here
$L_\be(x)\in[1+\delta,\frac1{\be-1}-\delta]$ and
$R_\be(x)\in[\delta,\frac1{\be-1}-\delta-1]$. In either case, the image is at a distance at least $\delta$ from either endpoint of $I_\be$.

It suffices to estimate the number of iterations one needs to reach the switch region $\De_\be$. In view of symmetry, we can deal with
$y\in[\delta,1/\be)$; here $L_\be^k(y)\in\De_\be$ for
some $1\leq k\le\lfloor\log_\be\frac1{\delta}\rfloor+1$.

\medskip\noindent
\textbf{Case 2.} $x\in\bigl(\frac1{\be},\frac{1+\delta}{\be}\bigr)$ or the mirror-symmetric case (which is analogous). Here $R_\be(x)$ can be very close to 0, so we have no control over its further iterations. Consequently, we remove this branch from $\T(x;\be)$ and concentrate on the subtree which grows from $L_\be(x)$.

We have $L_\be(x)=\be x\in(1,1+\delta)$. Clearly, it lies in $\De_\be$ provided $1+\delta\le \frac1{\be(\be-1)}$ -- which is equivalent to $\delta\le\frac{1+\be-\be^2}{\be(\be-1)}$, and this is true for $\be>\sqrt2$, in view of (\ref{eq:delta}); for $\be\le\sqrt2$ we have $\delta-1=\be\le\frac{1}{\be(\be-1)}$. Furthermore,
\begin{align*}
L_\be L_\be(x)&=\be^2 x \in (\be,(1+\delta)\be),\\
R_\be L_\be(x)&=\be^2 x-1 \in (\be-1,(1+\delta)\be-1).
\end{align*}
Notice that $\be^2 x\le\frac1{\be-1}-\delta$, because for $\be>\sqrt2$ we have
$\be(1+\delta)=\frac1{\be-1}-\delta$, and for $\be\le\sqrt2$ we have $\be^2\le\frac1{\be-1}-\be+1$, which is in fact equivalent to $\be\le\sqrt2$. As for $R_\be L_\be(x)$, it is clear that it lies in $\bigl(\delta,\frac1{\be-1}-\delta\bigr)$ in either case, since $\delta\le\be-1$.

\begin{figure}
\centering \scalebox{1.0} {\includegraphics{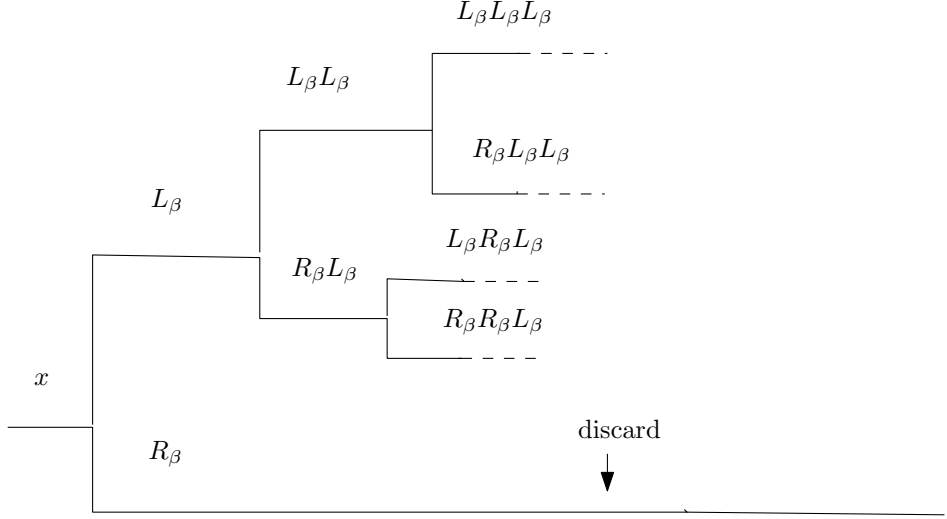}}
\caption{Branching for Case~2}\label{fig:branch1}
\end{figure}

We see that the length of each branch of the new tree does not
exceed $2(\lfloor\log_\be \frac1{\delta}\rfloor+1)$ (the factor two appears in the estimate because it may happen that we will have to discard $L_\be L_\be L_\be/R_\be L_\be L_\be$ or $L_\be R_\be L_\be/R_\be R_\be L_\be$), and it suffices to apply
Lemma~\ref{lem:estimate}.
\end{proof}

If $0<x<1/\be$, then there exists a unique $\ell\ge1$ such that $L_\be^{\ell-1}(x)<1/\be$ but $L_\be^{\ell}(x)\ge1/\be$. In view of $1<\frac1{\be(\be-1)}$, this implies $L_\be^{\ell}(x)\in\De_\be$. Similarly, if $\frac1{\be(\be-1)}<x<\frac1{\be-1}$, then $R_\be^{\ell}(x)\in\De_\be$ for some $\ell$. Thus, it takes only a finite number of iterations for any $x\in\bigl(0,\frac1{\be-1}\bigr)$ to reach $\Delta_\be$. Hence follows Theorem~\ref{thm-1.5}.


\begin{re}It would be interesting to determine the sharp analog of the golden ratio in Theorem~\ref{thm-1.5} for the case $\be\in(m-1,m)$ with $m\ge3$.
\end{re}

To prove Corollary~\ref{cor-1.5}, we apply the inequalities (\ref{eq:ineq1}) and (\ref{eq:Nnl}), whence
\[
\mu\left(x-\frac{\be^{-n}}{\be-1}, x+\frac{\be^{-n}}{\be-1}\right)\ge \frac12\cdot 2^{(1-\kappa)n},
\]
and
\[
\limsup_{n\to\infty} -\frac1n\log_\be\mu\left(x-\frac{\be^{-n}}{\be-1}, x+\frac{\be^{-n}}{\be-1}\right)\le(1-\kappa)\log_\be 2.
\]
Hence follows the claim of Corollary~\ref{cor-1.5}.

We define the {\em minimal growth exponent} as follows:
\[
\mathfrak m_\beta=\inf_{x\in (0,1/(\be-1))} \liminf_{n\to\infty}\sqrt[n]{\mathcal N_n(x;\be)}.
\]

\begin{cor}\label{cor:minge}
For $\be<\frac{1+\sqrt5}2$ we have $\mathfrak m_\beta\ge 2^{\kappa}>1$, where $\kappa$ is given by (\ref{eq:ga}).
\end{cor}


The golden ratio in Theorem~\ref{thm:lessgr} and Corollary~\ref{cor:minge} is a sharp constant in a boring sense, since for $\be>\frac{1+\sqrt5}2$ there are always $x$ with a unique $\be$-expansion (see \cite{GS}) and for $\be=\frac{1+\sqrt5}2$ there are $x$ with a linear growth of $\mathcal N_n(x)$ (see, e.g., \cite{SV}). Hence $\mathfrak m_\be=1$ for $\be\ge\frac{1+\sqrt5}2$.

However, it is also a sharp bound in a more interesting sense; let us call the set of $\be$-expansions of a given $x$ {\em sparse} if $\lim_{n\to \infty} \frac1n \log \mathcal N_n(x;\be)=0$.

\begin{pro} \label{prop:cont}
For $\be=\frac{1+\sqrt5}2$ there exists a continuum of points $x$, each of which has a sparse continuum of $\be$-expansions.
\end{pro}
\begin{proof} Suppose $(m_k)_{k=1}^\infty$ is a strictly increasing sequence of natural numbers. Let $x$ be the number whose $\be$-expansion is $10^{2m_1}10^{2m_2}10^{2m_3}\dots$ We claim that such an $x$ has a required property.

Indeed, as was shown in \cite{SV}, the set of all $\be$-expansions in this case is the Cartesian product $\mathfrak X_{m_1}\times \mathfrak X_{m_2}\times\dots$, where
\[
\mathfrak X_{m_k}=\left\{(\e_1,\dots,\e_{2m_k+1}) : \sum_{j=1}^{2m_k+1}\e_j\be^{-j}=\frac1{\be}\right\}.
\]
It follows from \cite[Lemma~2.1]{SV} that $\#\mathfrak X_{m_k}=m_k$, whence by \cite[Lemma~2.2]{SV},
\[
\#\mathcal D_\be(10^{2m_1}\dots 10^{2m_k})=\prod_{j=1}^k m_j,
\]
where $\mathcal D_\be(\cdot)$ is given by
\[
\mathcal D_\be(\e_1,\dots,\e_n)=\left\{(\e_1',\dots,\e_n')\in\{0,1\}^n : \sum_{k=1}^n \e_k\be^{-k}=\sum_{k=1}^n \e'_k\be^{-k}\right\}.
\]
Hence for $n=\sum_1^k (2m_j+1)$,
\[
\frac{\log\mathcal N_n(x;\be)}n \sim \frac{\sum_{j=1}^k \log m_j}{2\sum_1^k m_j+1}\to0, \quad k\to+\infty,
\]
since $m_k\nearrow +\infty$. Therefore, $\lim_n\sqrt[n]{\mathcal N_n(x;\be)}= 1$. It suffices to observe that there exists a continuum strictly increasing sequences of natural numbers -- for instance, one can always choose $m_k\in\{2k-1,2k\}$.
\end{proof}
A similar proof works for the multinacci $\be$. It is an open question whether given $\be>\frac{1+\sqrt5}2$, it is always possible to find $x$ with a sparse continuum of $\be$-expansions.

{\bf Acknowledgement.} The first author was partially supported by RGC grants (projects 400706 and 401008) in CUHK. The authors are indebted to Kevin Hare for his useful remarks regarding Section~\ref{sec:gr}.


\begin{thebibliography}{99}

\bibitem{DdV2} K. Dajani and M. de Vries, \textit{Measures of maximal entropy for random $\beta$-expansions}, J. Eur. Math. Soc. (JEMS) {\bf 7} (2005), 51–-68.

\bibitem{DdV} K. Dajani and M. de Vries,
\textit{Invariant densities for random $\beta$-expansions}, J. Eur. Math. Soc. {\bf 9} (2007), 157--176.

\bibitem{DST99} J. M. Dumont, N. Sidorov and A. Thomas,
\textit{Number of representations related to a linear recurrent basis}, Acta Arith. {\bf 88} (1999),  371--396.

\bibitem{EJK}
P. Erd\H os, I. Jo\'o, and V. Komornik,
\textit{Characterization of the unique expansions
$1=\sum_{i=1}^\infty q^{-n_i}$ and related problems}, Bull. Soc. Math. Fr. \textbf{118} (1990), 377--390.

\bibitem{Fen03}  D. J. Feng, \textit{Smoothness of the $L\sp q$-spectrum of self-similar measures with overlaps}, J. London Math. Soc. (2) {\bf 68} (2003), 102--118.

\bibitem{Fen05} D. J. Feng,  \textit{The limited Rademacher functions and Bernoulli convolutions associated with
Pisot numbers}, Adv. Math. {\bf 195} (2005), 24--101.

\bibitem{Fen07} D. J. Feng, \textit{Lyapunov exponent
for products of  matrices and  multifractal analysis.
 Part II: General matrices},  Israel J. Math. {\bf 170} (2009), 355--394.

 \bibitem{Ga0}A. Garsia, \textit{Arithmetic properties of Bernoulli convolutions}, Trans. Amer. Math. Soc. {\bf 102} (1962), 409--432.

\bibitem{Ga}A. Garsia, \textit{Entropy and singularity of
infinite convolutions}, Pac. J. Math. \textbf{13} (1963),
1159--1169.

\bibitem{GS} P. Glendinning and N. Sidorov, \textit{Unique
representations of real numbers in non-integer bases},
Math. Res. Lett. \textbf{8} (2001), 535--543.

\bibitem{GKT}P. Grabner, P. Kirschenhofer and T. Tichy, \textit{Combinatorial and arithmetical properties of linear numeration systems}, Combinatorica {\bf 22} (2002), 245--267.

\bibitem{HS}K. Hare and N. Sidorov, \textit{A lower bound for Garsia's entropy for certain Bernoulli convolutions}, preprint, see http://arxiv.org/abs/0811.3009.


\bibitem{Hut81} J. E. Hutchinson, \textit{Fractals and self-similarity}, Indiana Univ. Math. J. {\bf 30} (1981), 713--747.

\bibitem{Lalley} S. Lalley, \textit{Random series in powers of algebraic integers: Hausdorff dimension of the limit distribution}, J. London Math. Soc. {\bf 57} (1998), 629--654.

\bibitem{LaNg99}  K. S. Lau and S. M. Ngai, \textit{Multifractal measures and a weak
separation condition}, Adv. Math. {\bf 141} (1999), 45--96.


\bibitem{PrUr89}
F. Przytycki and M. Urba\'{n}ski, \textit{On the Hausdorff
dimension of some fractal sets},
Studia Math. {\bf 93} (1989), 155--186.


\bibitem{S} N. Sidorov,
\textit{Almost every number has a continuum of $\beta$-expansions}, Amer. Math. Monthly \textbf{110} (2003), 838--842.

\bibitem{S-nonl} N. Sidorov,
\textit{Combinatorics of linear iterated function systems with overlaps}, Nonlinearity  {\bf 20} (2007), 1299--1312.

\bibitem{SV}N. Sidorov and A. Vershik, \textit{Ergodic
properties of Erd\"{o}s measure, the entropy of the goldenshift, and related problems}, Monatsh. Math. \textbf{126} (1998), 215--261.

\bibitem{Sol}B. Solomyak, \textit{Notes on Bernoulli convolutions},  Proc. Symp. in Pure Math. {\bf 72}  (2004), 207--230, American Mathematical Society.


\bibitem{Wa}P. Walters, {\em An Introduction to Ergodic Theory}, Springer-Verlag, New York, 1982.



\end{thebibliography}
\end{document}